\documentclass{compositio}
\usepackage[latin1]{inputenc}
\usepackage{amsmath}
\usepackage{amssymb}
\usepackage{latexsym}
\usepackage{amscd}
\usepackage{mathrsfs}
\usepackage[all]{xy}
\usepackage{eufrak}
\usepackage{graphicx}
\usepackage[colorlinks=true,citecolor=red,linkcolor=blue]{hyperref}

\newcommand{\B}{\mathrm{B}}
\newcommand{\Ba}{\mathscr{B}}

\newcommand{\e}{\mathbf{e}}

\newcommand{\M}{\mathrm{M}}

\newcommand{\sm}[1]{\left(\begin{smallmatrix}#1\end{smallmatrix}\right)}

\newtheorem{theorem}{Theorem}[section]
\newtheorem{proposition}[theorem]{Proposition}
\newtheorem{lemma}[theorem]{Lemma}

\newtheorem{remark}[theorem]{\textsc{Remark}}

\numberwithin{equation}{section}

\makeindex

\begin{document}

\title[Small connections are cyclic]{Small connections are 
cyclic}

\author{Andrea Pulita}

\email{pulita@math.univ-montp2.fr}

\address{Departement de Mathématique, 
Université de Montpellier II, Bat 9, CC051, 
Place Eugène Bataillon, 34095 
Montpellier Cedex 05, France.}



\begin{abstract}
The main local invariants of a (one variable) differential 
module over the complex numbers are given by means of 
a cyclic basis. In the $p$-adic setting the existence of a 
cyclic vector is often unknown. 
We investigate the existence of such a cyclic vector in a 
Banach algebra. We follow the explicit method of 
Katz \cite{Katz-cyclic}, and we prove 
the existence of such a cyclic vector under the assumption 
that the matrix of the derivation is small enough in norm.
\if{ and we precise the upper bound on the norm. 
Then we investigate whether this assumption is satisfied in some basis. 
Since the matrix of the derivation is not an invariant of the differential 
module, in the case of a one dimensional affinoid space we provide a sufficient condition in terms 
of the radius of convergence function ensuring that the matrix of the derivation is small.}\fi 
\end{abstract}

\maketitle


\makeatletter
\renewcommand\tableofcontents{%
    \subsection*{\contentsname}%
    \@starttoc{toc}%
    }
\makeatother

\setcounter{tocdepth}{3} \tableofcontents

\setcounter{section}{0}

\section{Katz's simple algorithm for cyclic vectors}
Let $(\Ba,d)$ be a commutative ring $\Ba$ with unit, together with a 
derivation\footnote{i.e. a $\mathbb{Z}$-linear map satisfying the Leibnitz rule $d(ab)=ad(b)+d(a)b$} 
$d:\Ba\to\Ba$.
We denote by $\Ba^{d=0}:=\{b\in\Ba\textrm{ such that }d(b)=0\}$ the sub-ring of constants.
A differential module $\M$ is a free $\Ba$-module of finite rank together with an action of the derivation 
\begin{equation}
\nabla\;:\;\M\to\M
\end{equation}
i.e. a $\mathrm{Z}$-linear map satisfying $\nabla(bm)=d(b)m+b\nabla(m)$ for all $b\in\Ba$, $m\in\M$. 
A cyclic vector for $\M$ is an element $m\in\M$ such that the family 
$\{ m,\nabla(m),\nabla^2(m),\ldots,\nabla^{n-1}(m)\}$ is a basis of $\M$ over $\Ba$. Such a vector does not always 
exists. Namely if $d=0$ is the trivial derivation, then $\nabla$ is merely a $\Ba$-linear map and $(\M,\nabla)$ is a 
torsion module over the ring of polynomials $\Ba[X]$ where the action of $X$ on $\M$ is given by $\nabla$. 
There is another counterexample in the case in which $\Ba=\mathbb{F}_p(X)$ is a
functions field in characteristic $p>0$: let $\M:=\mathbb{F}_q[X]^n$, with $n>q=p^r$, 
together with the trivial connection 
$\nabla(f_1,\ldots,f_n)=(f_1',\ldots,f_n')$, then, since $d^q=0$, one has $\nabla^q=0$ so $\M$ does not 
have any cyclic vector. The same happens replacing $\mathbb{F}_q$ by a ring $A$ having a maximal ideal $\mathfrak{m}$ 
such that $A/\mathfrak{m}\cong\mathbb{F}_q$. The trivial connection of $A[X]^n$ (with respect to $d/dx$) 
can not admit a cyclic vector, since otherwise its reduction to $\mathbb{F}_q[X]$ would be cyclic too.
\subsection{Three cyclic vector theorems.}  P.Deligne provided the existence of such a cyclic vector 
for all differential modules over a field of characteristic $0$ with non trivial derivation 
(cf. \cite[Ch.II,Lemme 1.3]{Deligne-Reg-Sing}). 
\begin{theorem}[\protect{\cite[Ch.II, Lemme 1.3]{Deligne-Reg-Sing}}]\label{deligne}
Let $\Ba$ be a field of characteristic $0$, then all differential modules over $\Ba$ admit a cyclic vector.
\end{theorem}
Subsequently N.Katz generalized the result of Deligne 
providing the following simple explicit algorithm:
\begin{theorem}[(\cite{Katz-cyclic})]\label{theorem 1 of Katz}
Assume that there exists an element $t\in \Ba$ such that $d(t)=1$. 
Assume moreover that $(n-1)!$ is invertible in $\Ba$, and that $\Ba^{d=0}$ contains a field $k$ such that
\footnote{The symbol $\#k$ means the number of elements of $k$, or, if $k$ is infinite, its cardinality.} $\#k>n(n-1)$.
Let $a_0,a_1,\ldots,a_{n(n-1)}$ be $n(n-1)+1$ distinct elements of $k$, and let $\e:=\{e_0,\ldots,e_{n-1}\}\subset\M$ 
be a basis of $\M$ over $\Ba$. Then Zarisky locally on $Spec(\Ba)$ one of the vectors 
\begin{equation}\label{cyclic vector}
c(\e,t-a_i)\;:=\;\sum_{j=0}^{n-1}\frac{(t-a_i)^j}{j!}\sum_{k=0}^j(-1)^k\binom{j}{k}\nabla^k(e_{j-k})
\end{equation}
is a cyclic vector of $\M$.
\end{theorem}
\begin{theorem}[(\cite{Katz-cyclic})]\label{theorem 2 of Katz}
If $\Ba$ is a local $\mathbb{Z}[1/(n-1)!]$-algebra, and if $a\in\Ba^{d=0}$ is such that the maximal ideal of $\Ba$
contains $t-a$, then $c(\e,t-a)$ is a cyclic vector for $\M$. 
\end{theorem}
The arguments of the Katz's proofs are the following. 
We consider the polynomial ring $\Ba[X]$ and we extend the derivation of $\Ba$ by $d(X)=1$. 
We denote again by $\nabla$ the action of $d$ on $\M\otimes_{\Ba}\Ba[X]$ given by 
$\nabla\otimes\mathrm{Id}_{\Ba[X]}+\mathrm{Id}_{\M}\otimes d$.
Each element $c_0$ in $\M\otimes_{\Ba}\Ba[X]$ can be uniquely represented as 
$c_0:=\sum_{j\geq 0}c_{0,j} X^j$, with $c_{0,j}\in\M$ for $j=0,1,\ldots$. 
The derivatives $\nabla^i(c_0)$ of $c_0$ then have the same form $c_i:=\nabla^i(c_0)=
\sum_{j\geq 0}c_{i,j}X^j$, with $c_{i,j}=\sum_{k=0}^ik!\binom{j+k}{j}\binom{i}{k}\nabla^{i-k}(c_{0,j+k})$. 
\begin{equation}
\begin{scriptsize}
\begin{array}{rlrrrrrrrrrrr}
c_0\phantom{)}&=&c_{0,0}&+&c_{0,1}\cdot X&+&c_{0,2}\cdot 
X^2&+&\cdots&+&c_{0,n-1}\cdot X^{n-1}&+&\cdots\\
\nabla(c_0)&=&c_{1,0}&+&c_{1,1}\cdot X&+&c_{1,2}\cdot X^2&+&\cdots&+&c_{1,n-1}\cdot X^{n-1}&+&\cdots\\
\nabla^2(c_0)&=&c_{2,0}&+&c_{2,1}\cdot X&+&c_{2,2}\cdot X^2&+&\cdots&+&c_{2,n-1}\cdot X^{n-1}&+&\cdots\\
\cdots&&\cdots&&\cdots&&\cdots&&\cdots&&\cdots&&\cdots\\
\nabla^{n-1}(c_0)&=&c_{n-1,0}&+&c_{n-1,1}\cdot X&+&c_{n-1,2}\cdot X^2&+&\cdots&+&c_{n-1,n-1}\cdot 
X^{n-1}&+&\cdots\\
\cdots&&\cdots&&\cdots&&\cdots&&\cdots&&\cdots&&\cdots
\end{array}
\end{scriptsize}\end{equation}
The main point is now that, if $(n-1)!$ is invertible in $\Ba$, and if the degree (with respect to $X$) of $c_0$ 
is less or equal to $n-1$, 
then the $0$-components $\{c_{0,0},c_{1,0},\ldots,c_{n-1,0}\}$ of $\{c_0,\nabla(c_0),\ldots,\nabla^{n-1}(c_0)\}$ 
\emph{uniquely determine $c_0$}. In fact we have the inversion formula 
\begin{equation}
c_{0,j}\;:=\;\frac{1}{j!}\sum_{k=0}^j(-1)^{j-k}\binom{j}{k}\nabla^{j-k}(c_{k,0})\;,\quad j=0,\ldots,n-1\;.
\end{equation}
The idea is then to choose the $0$-components equal to the basis of $\M$: $c_{k,0}:=e_k$. 
We then obtain the vector \eqref{cyclic vector}:
\begin{equation}
c(\e,X)\;:=\;\sum_{j=0}^{n-1}\frac{X^j}{j!}\sum_{k=0}^j(-1)^k\binom{j}{k}\nabla^k(e_{j-k})\;.
\end{equation}
 This choice implies that the determinant of the base change 
is a polynomial $P(X)\in\Ba[X]$ verifying $P(0)=1$, because the matrix $H(X)\in M_n(\Ba[X])$ 
expressing $\{c_0,\nabla(c_0),\ldots,\nabla^{n-1}(c_0)\}$ in the basis $\e$ verifies $H(0)=\mathrm{Id}$. 
In other words $P(X)$ is invertible as a formal power series in $\Ba[[X]]$, so that $c_0$ is a cyclic 
vector for $\M\otimes_{\Ba}\Ba[[X]]$.

We now specialize $X$ into an element $t-a$ verifying $d(t-a)=1$, this guarantee that the 
specialization commutes with the action of the derivation. Let us come to the proof of the above results.
If $\Ba$ is local, and if $t-a$ belongs to the maximal ideal, 
then $P(t-a)$ is clearly invertible since it is of the form $P(t-a)=P(0)+(t-a)Q(t-a)=1+y$, 
with $y$ in the maximal ideal. 
\if{Alternatively, one specializes $\M\otimes_{\Ba}\Ba[[X]]\to\M\otimes_{\Ba}\Ba=\M$ by 
sending $X$ into $t-a$. Since $c_0$ is a cyclic vector for $\M\otimes_{\Ba}\Ba[[X]]$, then its specialization is a 
cyclic vector for $\M$.}\fi This proves theorem \ref{theorem 2 of Katz}. 
%
Notice that if $\Ba$ is a field of characteristic $0$, then $\Ba^{d=0}$ is an 
infinite field, hence there exists at least a constant $a\in\Ba^{d=0}$ such that $P(t-a)\neq 0$, 
this is enough to prove Deligne's Theorem \ref{deligne}.\footnote{Notice that Deligne does not ask for the existence of $t\in\Ba$ 
satisfying $d(t)=1$. But it is easy to reduce the general case to this one by replacing the non trivial 
derivation $d$ by $\widetilde{d}:=f\cdot d$, with $f:=d(t)^{-1}$, and then using Lemma \ref{change of derivation}.} 
%
Now we come to the proof of Theorem \ref{theorem 1 of Katz}. Katz proves that the ideal 
$\mathscr{I}$ 
of $\Ba$ generated by the values $\{P(t-a_i)\}_{i=0,\ldots,n(n-1)}$ is the unit ideal. He argues as follows. 
We observe that the polynomial $P(X)$ has degree $\leq n(n-1)$, since 
\begin{equation}
c_0\wedge\nabla(c_0)\wedge\ldots\wedge\nabla^{n-1}(c_0)\;=\;P(X)\cdot e_0\wedge e_1\wedge\cdots\wedge e_{n-1}
\end{equation}
and the $n$ vectors $c_0,\nabla(c_0),\ldots,\nabla^{n-1}(c_0)$ have all degree $\leq(n-1)$. So we write 
$P=\sum_{s=0}^{n(n-1)}r_sX^s$ and $P(t-a_i)=\sum_{s=0}^{n(n-1)}r_s(t-a_i)^s$. 
Now for $i\neq j$ one has $(t-a_i)-(t-a_j)=a_j-a_i\neq 0$ in $k$, 
so $(t-a_i)-(t-a_j)$ is invertible in $\Ba$. 
Hence the Van Der Monde matrix $V:=((t-a_i)^j)_{0\leq i,j\leq n(n-1)}$ is invertible 
because its determinant is $\prod_{0\leq i<j\leq n(n-1)}(a_j-a_i)$. This implies that the ideal $\mathscr{I}$ 
is equal to the ideal generated by the coefficients $r_0,\ldots,r_{n(n-1)}$.
\footnote{The ideal $\mathscr{I}$ is the set of linear combinations $\sum_{i=0}^{n(n-1)} 
b_iP(t-a_i)=\;^t\!w\cdot(P(t-a_i))_i$ with coefficients $b_i$ in $\Ba$. Since $V$ is invertible, 
for all vector $v$ with coefficients in $\Ba$, there exists $w$ such that $^tw\cdot V=\;^t\!v$ and reciprocally. 
So that any linear combination $^tw\cdot (P(t-a_i))_i$ of the family $\{P(t-a_i)\}_i$ is in fact a linear 
combination of the family $\{r_i\}_i$, because $^tw\cdot(P(t-a_i))_i=\,^t\!w\cdot V\cdot(r_i)_i=\,^t\!v\cdot(r_i)_i$, and reciprocally. So $\mathscr{I}$ is the 
ideal generated by the family $\{r_i\}_i$.} Since $r_0=1$, then $\mathscr{I}=\Ba$. This concludes the Katz's proofs.

\subsubsection{About the assumptions of Katz's Theorems.} 
The assumption about the existence of $t$ such that $d(t)=1$ is not completely 
constrictive. 
Indeed it is enough to assume the existence of an element $\widetilde{t}\in\mathscr{B}$ such that 
$d(\widetilde{t})=f$ is invertible in $\mathscr{B}$. Then we replace the derivation 
$d$ by $\widetilde{d}:=f^{-1}\cdot d$ in order to have $\widetilde{d}(\widetilde{t})=1$. 
We then consider the connection $\widetilde{\nabla}:=f^{-1}\cdot\nabla$ on $\M$, and 
we form the Katz's cyclic vector \eqref{cyclic vector} constructed from the data of 
$(\widetilde{d},\widetilde{t},\widetilde{\nabla})$. Then
\begin{lemma}\label{change of derivation}
The vector $c$ is a cyclic vector for the differential module $(\M,\nabla)$ over $(\mathscr{B},d)$ if and only if $c$ 
is a cyclic vector for  $(\M,f\cdot\nabla)$ over $(\mathscr{B},f\cdot d)$, for an arbitrary invertible element 
$f\in\mathscr{B}$.
\end{lemma}
\begin{proof}
It is enough to prove that if $c$ is cyclic with respect to $(\M,\nabla)$ then it is a cyclic vector with respect to 
$(\M,f\nabla)$. We have to prove that the base change matrix from the basis $\{c,\nabla(c),\ldots,\nabla^{n-1}(c)\}$ 
to the family $\{c,(f\nabla)(c),(f\nabla)^2(c),\ldots,(f\nabla)^{n-1}(c)\}$ is invertible. The Leibnitz rule of 
$\nabla$ gives the relation $\nabla\circ f=f\circ\nabla+d(f)$ where $f$ and $d(f)$ denote 
respectively the multiplication in $\M$ by $f\in\Ba$ and $d(f)\in\Ba$.
One sees then that $(f\nabla)^k=f^k\nabla^k+\sum_{0\leq i\leq k-1}\alpha_i(f)\nabla^i$, for convenient elements 
$\alpha_i(f)\in\Ba$. This implies that the base change matrix is triangular with $(1,f,f^2,\ldots,f^{n-1})$ in the 
diagonal.
\end{proof}

\begin{remark} 
The Katz's algorithm is not invariant under the above change of derivation. 
In other words the Katz's vector $c_0$ obtained from $(d,t,\nabla)$ does not 
coincide with the Katz's vector $\widetilde{c}_0$ constructed from 
$(\widetilde{d},\widetilde{t},\widetilde{\nabla})$.\footnote{Notice that 
$d(t)=\widetilde{d}(\widetilde{t})=1$. Once we change $d$ we also have to change $t$ in order to 
preserve  this relation.} 
If one of them is a cyclic vector, then it is simultaneously 
cyclic for $\nabla$ and $\widetilde{\nabla}$, thanks to the above lemma. 
But actually, in our knowledge, the fact that one of them is cyclic does 
not imply necessarily that the other is cyclic too.  
\end{remark}

\subsection{The Katz's base change matrix.}
We now investigate the explicit form of the base change matrix $H(X)$. For this we need to introduce some notation. 
If a basis $\e$ of $\M$ is fixed then we can associate to the $n$-times iterated connection 
$\nabla^n:=\nabla\circ\nabla\circ\cdots\circ\nabla$ a matrix 
$G_n=(g_{n;i,j})_{i,j=0,\ldots,n-1}\in M_n(\Ba)$ whose rows are the image of the basis $\e$ by $\nabla^n$: 
\begin{equation}  
\nabla^n(e_i):=\sum_{j=0}^{n-1}g_{n;i,j}\cdot e_j\;.
\end{equation}
\begin{proposition}\label{explicit H(X)}
The Katz's base change matrix $H(X)$ verifying $(\nabla^i(c_0(\e,X)))_i=H(X)(e_i)_i$ has the form
\begin{equation}\label{Katz's base change matrix}
H(X)\;:=\;H_0(X)+H_1(X)\cdot G_1+\cdots+H_{2n-2}(X)\cdot G_{2n-2}\;,
\end{equation}
where the matrices $H_s(X)$, $s=0,\ldots 2n-2$, all belong to $\mathbb{Z}[\frac{1}{(n-1)!}][X]$ and satisfy the following properties:
\begin{enumerate}
\item One has
\begin{equation}
H_0(X)\;=\;\left(\begin{smallmatrix}
1&X&\frac{X^2}{2}&\frac{X^3}{3!}&\cdots&\cdots&\frac{X^{n-1}}{(n-1)!}\\
0&1&X&\frac{X^2}{2}&\frac{X^3}{3!}&\cdots&\frac{X^{n-2}}{(n-2)!}\\
0&0&1&X&\frac{X^2}{2}&\cdots&\frac{X^{n-3}}{(n-3)!}\\
0&0&0&1&X&\cdots&\frac{X^{n-4}}{(n-4)!}\\
\cdots&\cdots&\cdots&\cdots&\cdots&\cdots&\cdots\\
\cdots&\cdots&\cdots&\cdots&\cdots&\cdots&\cdots\\
\cdots&\cdots&\cdots&\cdots&\cdots&\cdots&\cdots\\
0&0&0&\cdots&\cdots&1&X\\\\
0&0&0&0&\cdots&\cdots&1
\end{smallmatrix}
\right)
\end{equation}
\item If $H_s(X)=(h_{s;i,j}(X))_{i,j}$ then
\begin{equation}
h_{s;i,j}(X)\;=\;\alpha(s;i,j)\frac{X^{s+j-i}}{(s+j-i)!}
\end{equation}
with
\begin{equation}\label{alpha(s,i,j)}
\alpha(s;i,j)\;=\;
\epsilon_{s;i,j}\cdot\left[ 
\sum_{k=\max(0,s+j-(n-1))}^{\min(i,s)}(-1)^{s+k}\binom{s-k+j}{j}\binom{i}{k}\right]\in\mathbb{Z}
\end{equation}
where 
\begin{equation}
\epsilon_{s;i,j}\;=\;\left\{\begin{array}{rcl}
1&\textrm{ if }&(s,j)\in[0,n-1+i]\times[\max(0,i-s),\min(n-1,n-1+i-s)]\\
&&\\
0&\textrm{ if }&(s,j)\notin[0,n-1+i]\times[\max(0,i-s),\min(n-1,n-1+i-s)]
\end{array}\right.
\end{equation}
\if{$\varepsilon_{s;i,j}=1$ if $(s,j)\in[0,n-1+i]\times[\max(0,i-s),\min(n-1,n-1+i-s)]$, and 
$\epsilon_{s;i,j}=0$ otherwise.}\fi
\item In particular one has 
$h_{s;i,j}=0$ if $j-i$ does not belong to the interval $[\max(1-s,1-n),n-1-s]$.
\end{enumerate}
\end{proposition}
\begin{proof}
Applying $\nabla^i$ to the vector $c_0(\e,X)$, and re-summing by setting $s:=m+k-j$ one obtains
\begin{eqnarray}
\nabla^i(c_0(\e,X))&\;=\;&\sum_{m=0}^{n-1}\sum_{j=0}^{m}\sum_{k=0}^i(-1)^{m-j}\binom{m}{j}\binom{i}{k}
d^{i-k}(\frac{X^m}{m!})\nabla^{m-j+k}(e_j)\\
&=&\sum_{s=0}^{n-1+i}\quad\sum_{j=\max(0,i-s)}^{\min(n-1,n-1+i-s)}\alpha(s;i,j)\frac{X^{s+j-i}}{(s+j-i)!}
\nabla^s(e_j)\;,
\end{eqnarray}
where $\alpha(s;i,j)$ is
\begin{equation}
\alpha(s;i,j)\;:=\;\left[ 
\sum_{k=\max(0,s+j-(n-1))}^{\min(i,s)}(-1)^{s-k}\binom{s-k+j}{j}\binom{i}{k}\right]\in\mathbb{Z}\;.
\end{equation} 
In matrix form, if $H_s(X)=(h_{s;i,j}(X))_{i,j=0,\ldots,n-1}$, then 
\begin{eqnarray}
(\nabla^i(c_0(\e,X)))_i&\;=\;&\Bigl(\sum_{s=0}^{2n-2}H_s(X)G_s\Bigr)\cdot(e_i)_i
\;=\;\sum_{s=0}^{2n-2}\Bigl(H_s(X)G_s\Bigr)\cdot(e_i)_i\\
&\;=\;&\sum_{s=0}^{2n-2}\Bigl(\sum_{j=0}^{n-1}h_{s;i,j}(X)g_{s;j,k}\Bigr)_{i,k}\cdot(e_i)_i
\;=\;\sum_{s=0}^{2n-2}\Bigl(\sum_{k=0}^{n-1}
\sum_{j=0}^{n-1}h_{s;i,j}(X)g_{s;j,k}e_k\Bigr)_{i}\qquad\\
&\;=\;&\sum_{s=0}^{2n-2}\Bigl(
\sum_{j=0}^{n-1}h_{s;i,j}(X)(\sum_{k=0}^{n-1}g_{s;j,k}e_k)\Bigr)_{i}\;=\;
\sum_{s=0}^{2n-2}\Bigl(
\sum_{j=0}^{n-1}h_{s;i,j}(X)\nabla^s(e_j)\Bigr)_{i}\;\;.\qquad
\end{eqnarray}
So that $\nabla^i(c_0(\e,X))=\sum_{s=0}^{2n-2}\sum_{j=0}^{n-1}h_{s;i,j}(X)\nabla^s(e_j)$. This means that 
\begin{equation}
h_{s;i,j}(X)\;=\;\alpha(s;i,j)\cdot\frac{X^{s+j-i}}{(s+j-i)!}\;.
\end{equation}
\end{proof}
Below we write the first examples of $H(X)$ for $n=2,3,4,5$.
\begin{eqnarray*}
&&\\
\!\!\!\!\!\!\!\!\!\!\!\!\!\!\!\!\!\!\!\!\!\!\!\!\!\!\!\!\!\!\!\!\!\!\!\!\!\!\!\!\!\!\!\!\!\!\!\!
n=2&:&\\ 
H(X)&=&\sm{1&X\\0&1}+
\sm{-X&0\\0&X}G_1+
\sm{0&0\\-X&0}G_2\\
&&\nonumber\\
&&\nonumber\\
n=3&:&\\
H(X)&=&
\sm{1&X&\frac{X^2}{2!}
\\0&1&X\\\\
0&0&1}+
\sm{-X&-2\frac{X^2}{2}&0\\
0&-X&-\frac{X^2}{2}\\
0&0&2X}G_1+
\sm{\frac{X^2}{2}&0&0\\0&-2\frac{X^2}{2}&0\\0&-3X&\frac{X^2}{2}} G_2+
\sm{0&0&0\\\frac{X^2}{2}&0&0\\X&-2\frac{X^2}{2}&0\\}G_3+\sm{0&0&0\\0&0&0\\\frac{X^2}{2}&0&0\\}G_4\\
&&\nonumber\\
&&\nonumber\\
n=4&:&\\ 
H(X)&=&
\sm{1&X&\frac{X^2}{2!}&\frac{X^3}{3!}\\
0&1&X&\frac{X^2}{2!}\\
0&0&1&X\\
\\0&0&0&1}+
\sm{
-X&-2\frac{X^2}{2}&-3\frac{X^3}{3!}&0\\
0&-X&-2\frac{X^2}{2}&\frac{X^3}{3!}\\
0&0&-X&2\frac{X^2}{2}\\
0&0&0&3X\\
}G_1+
\sm{
\frac{X^2}{2}&3\frac{X^3}{3!}&0&0\\
0&\frac{X^2}{2}&-3\frac{X^3}{3!}&0\\
0&0&-5\frac{X^2}{2}&\frac{X^3}{3!}\\
0&0&-6X&3\frac{X^2}{2}\\
}G_2+\\
&&\sm{
-\frac{X^3}{3!}&0&0&0\\
0&3\frac{X^3}{3!}&0&0\\
0&4\frac{X^2}{2}&-3\frac{X^3}{3!}&0\\
0&4X&-8\frac{X^2}{2}&\frac{X^3}{3!}\\
}G_3+
\sm{0&0&0&0\\
-\frac{X^3}{3!}&0&0&0\\
-\frac{X^2}{2}&3\frac{X^3}{3!}&0&0\\
-X&7\frac{X^2}{2}&-3\frac{X^3}{3!}&0
}G_4+
\sm{0&0&0&0\\0&0&0&0\\
-\frac{X^3}{3!}&0&0&0\\-2\frac{X^2}{2}&3\frac{X^3}{3!}&0&0\\}G_5+
\sm{0&0&0&0\\0&0&0&0\\0&0&0&0\\-\frac{X^3}{3!}&0&0&0\\}G_6\\
&&\nonumber\\
&&\nonumber\\
n=5&:&\\ 
H(X)&=&\sm{1&X&\frac{X^2}{2!}&\frac{X^3}{3!}&\frac{X^4}{4!}\\
0&1&X&\frac{X^2}{2!}&\frac{X^3}{3!}\\0&0&1&X&\frac{X^2}{2}\\
0&0&0&1&X&\\
\\0&0&0&0&1}+
\sm{
-X&-2\frac{X^2}{2!}&-3\frac{X^3}{3!}&-4\frac{X^4}{4!}&0\\
0&-X&-2\frac{X^2}{2!}&-3\frac{X^3}{3!}&\frac{X^4}{4!}\\
0&0&-X&-2\frac{X^2}{2!}&2\frac{X^3}{3!}\\
0&0&0&-X&3\frac{X^2}{2!}\\
0&0&0&0&4X}G_1+
\sm{
\frac{X^2}{2!}&3\frac{X^3}{3!}&6\frac{X^4}{4!}&0&0\\
0&\frac{X^2}{2!}&3\frac{X^3}{3!}&-4\frac{X^4}{4!}&0\\
0&0&\frac{X^2}{2!}&-7\frac{X^3}{3!}&\frac{X^4}{4!}\\
0&0&0&-9\frac{X^2}{2!}&3\frac{X^3}{3!}\\
0&0&0&-10X&6\frac{X^2}{2!}}G_2+\\
&&
\sm{-\frac{X^3}{3!}&-4\frac{X^4}{4!}&0&0&0\\
0&-\frac{X^3}{3!}&6\frac{X^4}{4!}&0&0\\
0&0&9\frac{X^3}{3!}&-4\frac{X^4}{4!}&0\\
0&0&10\frac{X^2}{2!}&-11\frac{X^3}{3!}&\frac{X^4}{4!}\\
0&0&10X&-20\frac{X^2}{2!}&4\frac{X^3}{3!}
}G_3+
\sm{\frac{X^4}{4!}&0&0&0&0\\
0&-4\frac{X^4}{4!}&0&0&0\\
0&-5\frac{X^3}{3!}&6\frac{X^4}{4!}&0&0\\
0&-5\frac{X^2}{2!}&15\frac{X^3}{3!}&-4\frac{X^4}{4!}&0\\
0&-5X&25\frac{X^2}{2!}&-13\frac{X^3}{3!}&\frac{X^4}{4!}\\
}G_4+\\&&
\sm{0&0&0&0&0\\
\frac{X^4}{4!}&0&0&0&0\\
\frac{X^3}{3!}&-4\frac{X^4}{4!}&0&0&0\\
\frac{X^2}{2!}&-9\frac{X^3}{3!}&6\frac{X^4}{4!}&0&0\\
X&-14\frac{X^2}{2!}&21\frac{X^3}{3!}&-4\frac{X^4}{4!}&0\\
}G_5+
\sm{0&0&0&0&0\\0&0&0&0&0\\
\frac{X^4}{4!}&0&0&0&0\\
2\frac{X^3}{3!}&-4\frac{X^4}{4!}&0&0&0\\
3\frac{X^2}{2!}&-13\frac{X^3}{3!}&6\frac{X^4}{4!}&0&0\\
}G_6+
\sm{0&0&0&0&0\\0&0&0&0&0\\0&0&0&0&0\\
\frac{X^4}{4!}&0&0&0&0\\
3\frac{X^3}{3!}&-4\frac{X^4}{4!}&0&0&0\\
}G_7+\\
&&\sm{0&0&0&0&0\\0&0&0&0&0\\0&0&0&0&0\\0&0&0&0&0\\
\frac{X^4}{4!}&0&0&0&0\\}G_8
\end{eqnarray*}
\section{Small connections are cyclic over an ultrametric Banach algebra.}
In this section we provide a sufficient condition for differential modules over an ultrametric Banach algebras, 
in order to guarantee that the Katz's vector \eqref{cyclic vector} is a cyclic vector. 

\subsection{Norms and matrices} 
We recall that an ultrametric norm on a commutative ring with unit $\Ba$ is a map 
$|.|:\Ba\to\mathbb{R}_{\geq 0}$ satisfying $|0|=0$, $|1|=1$, $|a+b|\leq\max(|a|,|b|)$, 
$|ab|\leq|a|\cdot|b|$, for all $a,b\in \Ba$. We require moreover $|na|=|n||a|$ for all $n\in\mathbb{Z}$, $a\in\Ba$.  
Hence, in particular, the norm on $\mathbb{Z}$ induced by $|.|$ is ultrametric and so $|n|\leq 1$ for all 
$n\in\mathbb{Z}$. If $\Ba$ is complete an separated\footnote{$\Ba$ is separated if and only if $|a|=0$ implies $a=0$.} 
with respect to $|.|$ then we say that $\Ba$ is 
an \emph{ultrametric Banach algebra}. 
A norm on $M_n(\Ba)$ is a map $\|.\|:M_n(\Ba)\to\mathbb{R}_{\geq 0}$ satisfying $\|0\|=0$, $\|1\|=1$, 
$\|A+B\|\leq\max(\|A\|,\|B\|)$, $\|AB\|=\|A\|\cdot\|B\|$, $\|bA\|=|b|\|A\|$ for all $b\in\Ba$, $A,B\in M_n(\Ba)$. 
In the sequel we will consider on $M_n(\Ba)$ two norms 
\begin{eqnarray}
\textrm{sup-norm: }
\qquad|(a_{i,j})|\phantom{^{(\rho)}}&\;:=\;&\sup_{i,j}|a_{i,j}|\;,\label{norm sup on matrices}\\
\rho\textrm{-sup-norm: }
\qquad|(a_{i,j})|^{(\rho)}&\;:=\;&\sup_{i,j}|a_{i,j}|\rho^{j-i}\;,\qquad\qquad\rho>0\;.\label{norm rho on 
matrices}
\end{eqnarray}
Notice that if $\mathscr{C}$ is a $\Ba$-algebra together with a norm $|.|_{\mathscr{C}}$ 
extending\footnote{i.e. in order that the structural morphism $\Ba\to\mathscr{C}$ is an isometry} 
that of $\Ba$, and if $c\in\mathscr{C}$ is an element with norm $|c|=\rho^{-1}$, 
then $|A|^{(\rho)}=|\Lambda_c^{-1}A\Lambda_c|$, for all $A\in M_n(\Ba)$, where $\Lambda_c$ is 
the diagonal matrix with diagonal equal to $(1,c,c^2,\ldots,c^{n-1})$.
 
\if{
If $\mathscr{C}$ is a $\Ba$-algebra together with a norm $|.|_{\mathscr{C}}$ 
extending\footnote{i.e. in order that the structural morphism $\Ba\to\mathscr{C}$ is an isometry} 
that of $\Ba$, then $GL_n(\mathscr{C})$ acts naturally on the set $\mathscr{M}(M_n(\Ba))$ of norms on $M_n(\Ba)$. The 
action is given by 
\begin{equation}
|.|^{(L)}(A)\;:=\;|L^{-1}AL|\;,
\end{equation}

----------

NON E' COSI FACILE DA DIRE... 
PERCHE' bisogna mostrare prima che la norma $|.|\in\mathscr{M}(M_n(\Ba))$ che sto considerando 
qui si estende a $\mathscr{C}$.

----------

where $L\in GL_n(\mathscr{C})$, and $A\in M_n(\Ba)$. In fact the norm $|.|^{(\rho)}$ 
(cf. \eqref{norm rho on matrices}) is in the orbit of $|.|$ (cf. \eqref{norm sup on matrices}), because if 
$c\in\mathscr{C}$ is an element with norm $|c|=\rho^{-1}$, then $|A|^{(\rho)}=|A|^{(\Lambda_c)}$ where $\Lambda_c$ is 
the diagonal matrix with $(1,c,c^2,\ldots,c^{n-1})$ in the diagonal. More generally we notice that all ring 
endomorphisms $\phi:\mathscr{C}\to\mathscr{C}$ gives rise to a norm by $|A|^{(\phi)}:=|\phi(A)|$.

----------

STESSA OSSERVAZIONE QUI...

----------
}\fi

\subsubsection{Norm of derivation} 
Let $(\B,|.|)$ be an ultrametric Banach algebra, and let $\|.\|:\M_n(\Ba)\to\mathbb{R}_{\geq 0}$ be a fixed norm. 
Let now $d:\Ba\to\Ba$ be a continuous derivation. We extend $d$ to $M_n(\Ba)$ by 
$d((a_{i,j})_{i,j}):=(d(a_{i,j}))_{i,j}$. Let $|d|$ denotes the norm operator of $d$ acting on $\Ba$: 
\begin{equation}
|d|\;:=\;\sup_{b\neq 0,b\in\Ba}\frac{|d(b)|}{|b|}\;.
\end{equation}
We will always assume that the norm $\|.\|$ verifies 
\begin{equation}\label{compatibility with d}
\|d(A)\|\;\leq\;|d|\cdot\|A\|
\end{equation} for all $A\in M_n(\Ba)$. This holds for the sup-norm and the $\rho$-sup-norm.
\subsection{Norm of the matrix of the connection and cyclic vectors}
We consider as above an ultrametric Banach algebra $(\Ba,|.|)$, together with a continuous derivation $d:\Ba\to\Ba$.
Let $\|.\|:\M_n(\Ba)\to\mathbb{R}_{\geq 0}$ be a fixed norm 
satisfying \eqref{compatibility with d}, for which 
$M_n(\Ba)$ is complete and separated. Let $(\M,\nabla)$ be a differential module. We assume 
that there is a element $t\in\Ba$ such that $d(t)=1$. 
In order to consider the Katz's base change matrix \eqref{Katz's base change matrix} we assume that $(n-1)!$ is 
invertible in $\Ba$.
As in the above sections we denote by $G_n$ the matrix of the $n$-th iterated 
connection $\nabla^n:\M\to\M$ with respect to a basis $\e$. The simple idea of this section is the following. 
\begin{lemma}
If the matrices 
$G_1,\ldots,G_{2n-2}$ are small enough in norm, in order to verify 
\begin{equation}\label{small norm implies invertible}
\|H_0(-t)H_s(t)G_s\|<1
\end{equation} for all $s=1,\ldots,2n-2$, 
then the Katz's base change matrix 
\begin{equation}
H(t)\;:=\;H_0(t)+H_1(t)G_1+\cdots+H_{2n-2}G_{2n-2}
\end{equation} 
is invertible. 
\end{lemma}
\begin{proof}
Indeed $H_0(t)$ is always invertible with inverse 
\begin{equation}
H_0(t)^{-1}\;=\;H_0(-t)\;.
\end{equation} 
So that $H(t)$ is invertible if and only if 
$H_0(t)^{-1}H(t)=1+\sum_{s=1}^{2n-2}H_0(-t)H_s(t)G_s$ is invertible. 
\end{proof} 
Of course a sufficient condition to have \eqref{small norm implies invertible} is 
\begin{equation}\label{condition ....}
\|G_s\|<(\|H_s(t)\|\cdot\|H_0(-t)\|)^{-1}\;.
\end{equation}
In the following subsection we provide an explicit upper bound on the sup-norm 
and on the $\rho$-sup-norm of 
$G:=G_1$ sufficient to guarantee \eqref{condition ....} for all $s=1,\ldots,2n-2$. In order to do that we relate the 
norm of $G_s$ with that of $G_1$ by the following
\begin{lemma}\label{G_s+1 les G_1^s}
For all $s\geq 1$ one has
\begin{equation}
\|G_s\|\;\leq\;\|G_1\|\cdot\sup(\|G_1\|,|d|)^{s-1}\;.
\end{equation}
\end{lemma}
\begin{proof}
We have the recursive relation $G_{s+1}=d(G_s)+G_sG_1$. Since we assume $\|d(G_s)\|\leq|d|\|G_s\|$, then one easily has 
$\|G_{s+1}\|\leq \|G_s\|\cdot\max(|d|,\|G_1\|)$. By induction the lemma is proved.
\end{proof}

\subsection{Upper bound for the sup-norm.} 
Let now the chosen norm $\|.\|=|.|$ be the sup-norm 
\eqref{norm sup on matrices}. We are looking for a condition on $|G_1|$ that guarantee 
\begin{equation}\label{condition -t-}
|G_s|<(|H_s(t)|\cdot|H_0(-t)|)^{-1}\;,
\end{equation}
for all $s=1,\ldots,2n-2$. Thanks to Proposition \ref{explicit H(X)} one 
has\footnote{Notice that $|.|$ is not assumed 
to be multiplicative, hence $|t^i|\leq|t|^i$.}
\begin{eqnarray}
|H_0(t)|&\;=\;&|H_0(-t)|=\sup_{i=0,\ldots,n-1}|t^i|/|i!|\\
|H_s(t)|&\leq&|H_0(-t)|\;,\textrm{ for all }s=1,\ldots,2n-2\;.
\end{eqnarray}
Indeed since $\alpha(s;i,j)$ is an integer, and since the norm $|.|$ is ultrametric, one has $|\alpha(s;i,j)|\leq 1$. 
From this we have 
\begin{equation}
(|H_0(-t)||H_s(t)|)^{-1}\;\geq\;|H_0(t)|^{-2}\;.
\end{equation} 
On the other hand by Lemma \ref{G_s+1 les G_1^s} one has
\begin{equation}
|G_s|\;\leq\;|G_1|\cdot\max(|G_1|,|d|)^{s-1}\;.
\end{equation}
Hence it is enough to prove that
\begin{equation}\label{sufficient condition ro}
|G_1|\cdot\max(|G_1|,|d|)^{s-1} \; < \;|H_0(t)|^{-2}\;,
\end{equation}
for all $s=1,\ldots,2n-2$. 
\begin{proposition}
Assume that
\begin{equation}
|G_1|\;<\;|H_0(t)|^{-2}\cdot 
\min\Bigl(1,\frac{1}{|d|^{2n-3}}\Bigr)\;\;=\;\;
\min\Bigl(1,\frac{1}{|t|},\frac{|2|}{|t^2|},\ldots,\frac{|(n-1)!|}{|t^{n-1}|}\Bigr)^2\cdot
\min\Bigl(1,\frac{1}{|d|^{2n-3}}\Bigr)\;.
\end{equation}
Then $(\M,\nabla)$ is cyclic and the Katz's vector $c_0(\e,t)$ is a cyclic vector for $\M$.
\end{proposition}
\begin{proof}
We observe that both minimums are $\leq 1$, moreover 
$\min(1,1/|t|,|2|/|t^2|,\ldots,|(n-1)!|/|t^{n-1}|)\leq1/|t|$. 
Since $d(t)=1$, then $|d||t|\geq 1$, and hence 
$1/|t|\leq|d|$. Our assumption then implies $|G_1|<|d|$. Hence 
\eqref{sufficient condition ro} becomes $|G_1|\cdot|d|^{s-1}<|H_0(t)|^{-2}$ for all 
$s=1,\ldots,2n-2$. This inequality is fulfilled if and only if 
$|G_1|< \min_{s=1,\ldots,2n-2}|H_0(t)|^{-2}/|d|^{s-1}= 
|H_0(t)|^{-2}\cdot\min(1,1/|d|^{2n-3})$ which 
is our assumption.
\end{proof}

\subsection{Upper bound for the $\rho$-sup-norm with $\rho=|t|^{-1}$.}\label{section upper bound for rho=t}
We assume that 
\begin{equation}
\rho\;:=\;|t|^{-1}\;,\qquad|(a_{i,j})_{i,j}|^{(|t|^{-1})}\;=\;\sup_{i,j}|a_{i,j}||t|^{i-j}\;.
\end{equation}
As above we shall provide a condition on $G_1$ to guarantee
\begin{equation}\label{condition rho}
|G_s|^{(|t|^{-1})}<(|H_s(t)|^{(|t|^{-1})}\cdot|H_0(-t)|^{(|t|^{-1})})^{-1}\;,
\end{equation}
for all $s=1,\ldots,2n-2$. Since $|.|$ is not assumed to be multiplicative, hence $|t^i|\leq|t|^i$. This implies
\begin{equation}
|H_0(-t)|^{(|t|^{-1})}=|H_0(t)|^{(|t|^{-1})}=\sup_{i,j=0,\ldots,n-1}\frac{|t^i||t|^{-i}}{|i!|}
\;\leq\;\sup_{i=0,\ldots,n-1}\frac{1}{|i!|}\;=\;\frac{1}{|(n-1)!|}\;.
\end{equation}
Of course if $|.|$ is power multiplicative\footnote{The norm $|.|$ is power multiplicative if it verifies $|b^n|=|b|^n$ 
for all $b\in\Ba$, and all integer $n\geq 0$}, the above inequality is actually an equality.
Notice that since $|.|$ is ultrametric on $\mathbb{Z}$, then $|(n-1)!|\leq 1$. 
\begin{lemma}\label{Lemma H_s for norm rho=t^-1}
One has 
\begin{equation}
|H_s(t)|^{(|t|^{-1})}\;\leq\;\frac{|t^s|}{|(n-1)!|}\;.
\end{equation}
\end{lemma}
\begin{proof}
Thanks to proposition \ref{explicit H(X)}, for all $s=1,\ldots,2n-2$ one has 
\begin{equation}
|H_s(t)|^{(|t|^{-1})}\;=\;
\max_{i,j=0,\ldots,n-1}|h_{s;i,j}||t|^{i-j}=\max_{i,j=0,\ldots,n-1}|\alpha(s;i,j)|\frac{|t^{s+j-i}|}{|(s+j-i)!|}
|t|^{i-j}\;.
\end{equation}
Now $|\alpha(s,i,j)|\leq 1$, and it is equal to $0$ for $j-i\notin[\max(1-s,1-n),n-1-s]$. So, since 
$|t^{s+j-i}|\leq|t^s||t|^{j-i}$, then we obtain
$|H_s(t)|^{(|t|^{-1})}\;\leq\;
\max_{j-i\in[\max(1-s,1-n),n-1-s]}\frac{|t^{s+j-i}|}{|(s+j-i)!|}|t|^{i-j}
\leq\max_{r\in[\max(1-s,1-n),n-1-s]}\frac{|t^s|}{|(s+r)!|}\;=\;\frac{|t^s|}{|(n-1)!|}$.
\end{proof}

Then one has 
\begin{equation}
(|H_0(-t)|^{(|t|^{-1})}\cdot|H_s(t)|^{(|t|^{-1})})^{-1}
\;\geq\; \frac{|(n-1)!|^2}{|t|^s}\;.
\end{equation}
On the other hand by Lemma \ref{G_s+1 les G_1^s} one has
\begin{equation}
|G_s|^{(|t|^{-1})}\;\leq\;|G_1|^{(|t|^{-1})}\cdot\max(|G_1|^{(|t|^{-1})},|d|)^{s-1}\;.
\end{equation}
So condition \eqref{condition rho} is fulfilled if
\begin{equation}\label{sufficient condition roughhht}
|G_1|^{(|t|^{-1})}\cdot\max(|G_1|^{(|t|^{-1})},|d|)^{s-1} \; < \;\frac{|(n-1)!|^2}{|t|^s}
\end{equation}
for all $s=1,\ldots,2n-2$. 
\begin{proposition}
Assume that
\begin{equation}
|G_1|^{(|t|^{-1})}\;<\;\frac{|(n-1)!|^2|d|}{(|d||t|)^{2n-2}}\;.
\end{equation}
Then $(\M,\nabla)$ is cyclic and the Katz's vector $c_0(\e,t)$ is a cyclic vector for $\M$.
\end{proposition}
\begin{proof}
Since $d(t)=1$, then $|d||t|\geq 1$. Our assumption then implies $|G_1|^{(|t|^{-1})}<|(n-1)!|^2|d|\leq|d|$. Hence 
\eqref{sufficient condition roughhht} becomes $|G_1|^{(|t|^{-1})}\cdot|d|^{s-1}<\frac{|(n-1)!|^2}{|t|^s}$ for all 
$s=1,\ldots,2n-2$. This inequality is fulfilled if and only if 
$|G_1|^{(|t|^{-1})}<\min_{s=1,\ldots,2n-2}\frac{|(n-1)!|^2|d|}{(|t||d|)^s}= 
\frac{|(n-1)!|^2|d|}{(|t||d|)^{2n-2}}$ which 
is our assumption.
\end{proof}

\subsection{Upper bound for the $\rho$-sup-norm with $\rho=|d|$.}
We now set 
\begin{equation}
\rho\;:=\;|d|\;,\qquad|(a_{i,j})_{i,j}|^{(|d|)}\;=\;\sup_{i,j}|a_{i,j}||d|^{j-i}\;.
\end{equation}
We quickly reproduce the computations of section \ref{section upper bound for rho=t}. As usual we have to prove that
$|G_s|^{(|d|)}<(|H_0(-t)|^{(|d|)}\cdot|H_s(t)|^{(|d|)})^{-1}$. One has 
\begin{equation}
|H_0(-t)|^{(|d|)}\;=\;|H_0(t)|^{(|d|)}\;=\;\max_{i=0,\ldots,n-1}\frac{|t^i||d|^{i}}{|i!|}\;\leq\;
\max_{i=0,\ldots,n-1}\frac{(|d||t|)^{i}}{|i!|}
\end{equation}
As usual this becomes an equality if $|.|$ is power multiplicative.
\begin{lemma}\label{increasing if rho>1}
Let $\rho\geq 1$ be a real number, and let $s\geq 0$ be an integer. The sequence of real numbers $i\mapsto 
\rho^i/|(s+i)!|$ is increasing.
\end{lemma}
\begin{proof}
One has $\rho^{i+1}/|(s+i+1)!|\geq\rho^i/|(s+i)!|$ if and only if $\rho/|s+i+1|\geq 1$. This last is true since the 
norm of a integer is $\leq 1$, because the norm is ultrametric.
\end{proof}
Since $d(t)=1$, then $|d||t|\geq 1$, so we then have
\begin{equation}
|H_0(-t)|^{(|d|)}\;\leq\;\frac{(|d||t|)^{n-1}}{|(n-1)!|}\;.
\end{equation}
\begin{lemma}
One has
\begin{equation}
|H_s(t)|^{(|d|)}\;\leq\;|t^s|\cdot\frac{(|d||t|)^{n-1-s}}{|(n-1)!|}
\end{equation}
\end{lemma}
\begin{proof}
As in lemma \ref{Lemma H_s for norm rho=t^-1} one has 
\begin{equation}
|H_s(t)|^{(|d|)}=\max_{i,j}|\alpha(s;i,j)|\frac{|t^{s+j-i}||d|^{j-i}}{|(s+j-i)!|}\leq
\max_{i,j}\frac{|t^{s}|(|d||t|)^{j-i}}{|(s+j-i)!|}=
|t^{s}|\cdot\max_{r}\frac{(|d||t|)^{r}}{|(s+r)!|}\;,
\end{equation}
where $i,j$ runs in $[0,n-1]$, and $r\in[\max(1-s,1-n),n-1-s]$. By Lemma \ref{increasing if rho>1} the last maximum is 
equal to $(|d||t|)^{n-1-s}/|(n-1)!|$.
\end{proof}
Then one has
\begin{equation}
(|H_0(-t)|^{(|d|)}\cdot|H_s(t)|^{(|d|)})^{-1}\;\geq\; \frac{|(n-1)!|^2}{|t^s|\cdot(|d||t|)^{2n-2-s}}\;.
\end{equation}
As usual one also has $|G_s|^{(|d|)}\leq|G_1|^{(|d|)}\cdot\max(|G_1|^{(|d|)},|d|)^{s-1}$, 
so what we need is
\begin{equation}\label{drtdrtdrt}
|G_1|^{(|d|)}\cdot\max(|G_1|^{(|d|)},|d|)^{s-1} \; < \;\frac{|(n-1)!|^2}{|t|^s(|d||t|)^{2n-2-s}}
\end{equation}
for all $s=1,\ldots,2n-2$. 
\begin{proposition}
Assume that
\begin{equation}
|G_1|^{(|d|)}\;<\;\frac{|(n-1)!|^2|d|}{(|d||t|)^{2n-2}}\;.
\end{equation}
Then $(\M,\nabla)$ is cyclic and the Katz's vector $c_0(\e,t)$ is a cyclic vector for $\M$.
\end{proposition}
\begin{proof}
Since $d(t)=1$, then $|d||t|\geq 1$. Our assumption then implies $|G_1|^{(|t|^{-1})}<|(n-1)!|^2|d|\leq|d|$. Hence 
\eqref{drtdrtdrt} becomes $|G_1|^{(|d|)}\cdot|d|^{s-1}<\frac{|(n-1)!|^2}{|t|^{s}(|d||t|)^{n-1-s}}$, for all 
$s=1,\ldots,2n-2$. But this is actually our assumption.
\end{proof}

\if{
\section{Differential equations over an affinoid. 
Reduction to a small connection.}

Cher Gilles,\\

Je liste ici ce qu'il manque pour completer.
\begin{itemize}
\item Il nous faut des lemmes qui permettent de se ramener au cas de norme de la matrice de la connexion petite. Je crois qu'il faudra faire des hypothèses sur le rayon de convergence.
\item Ensuite si cela est vrai, j'ai le procédé qu'il faut pour montrer que toute équation différentielle
soluble avec structure de Frobenius sur l'anneau de Robba est cyclique et donc \emph{bornée} grâce à Young.
Cela est en effet déjà connu car Y.André a montré que,
grâce au théorème de la mondromie locale $p$-adique, ces modules ont un réseau sur l'anneau
d'Amice. Ce qui est intéressant ici est le fait que nous on y arrive par un procédé indépendant qui ne fait pas appel à
la ``\emph{grosse machine}''. L'idée est alors que cela devrait
permettre de trouver une preuve plus brève, ou alternative, du théorème de la monodromie locale $p$-adique,
car il est bien connu que l'une des difficultés majeures dans la théorie est que la matrice de la connexion
n'est pas bornée.
\item
\end{itemize}

\if{
\section{Da fare}
COSA SUCCEDE NEI CASI SEGUENTI: \\

1) CASO DI UN'ALGEBRA DI BANACH NON ULTRAMETRICA ?\\

2) CASO IN CUI IL CORPO DI BASE ABBIA TOPOLOGIA DISCRETA, e VALORE ASSOLUTO TRIVIALE ?\\

3) CASO DELL'ANELLO $K[[t]]$ tutto è cyclico qui. Ma il mio algoritmo cosa dice? \\

4) IL MIO ALGORITMO COSA DICE NEL CASO DI UN ANELLO TOPOLOGICO PIU' GENERALE? Servirebbe forse che l'anello topologico 
$\mathscr{B}$ avesse una topologia tale che valgano alcuni lemmi elementari. Forse basta che se \\

CITARE  : KATZ, DELIGNE, CHRISTOL-DWORK.

}\fi
}\fi

\bibliographystyle{amsalpha}
\bibliography{bib}

\end{document}